\newcommand\Modified{Aug 14, 2013}
\documentclass[11pt]{amsart}

\usepackage{amsfonts, amsmath,amscd, amssymb, latexsym, mathrsfs, slashed, stmaryrd, verbatim, wasysym }
\usepackage[all]{xy}
\usepackage{hyperref}

\newcommand\datver[1]{\def\datverp%
{\par\boxed{\boxed{\text{Version: #1; Run: \today}}}}}

\newtheorem{theorem}{Theorem}
\newtheorem{definition}[theorem]{Definition}

\newtheorem{example}{Example}
\newtheorem{lemma}[theorem]{Lemma}

\newtheorem{proposition}[theorem]{Proposition}
\newtheorem{remark}[theorem]{Remark}

\numberwithin{equation}{section}
\numberwithin{theorem}{section}

 \usepackage{color}
\definecolor{darkgreen}{cmyk}{1,0,1,.2}
\definecolor{m}{rgb}{1,0.1,1}
\renewcommand{\bar}{\overline}

\renewcommand{\hat}[1]{\widehat{#1}}
\newcommand{\rest}[1]{\big\rvert_{#1}} % restriction e.g. to boundary

\newcommand{\wt}[1]{\widetilde{#1}}

\newcommand\eps\varepsilon

\newcommand{\IC}[1]{\operatorname{\mathbf{IC}}_{\overline{#1}}^{\bullet}}

\renewcommand{\sc}[1]{\mathbf{ #1 }^{\bullet}} %sheaf complex
\newcommand\lra{\longrightarrow}
\newcommand\xlra[1]{\xrightarrow{\phantom{x} #1 \phantom{x}}}

\newcommand\iie{\operatorname{iie}}

\newcommand\Iie{{}^{\iie}}

\newcommand\CI{{\mathcal{C}}^{\infty}}
\newcommand\CIc{{\mathcal{C}}^{\infty}_c}

\newcommand{\lrpar}[1]{\left( #1 \right)}

\newcommand{\even}{\operatorname{even}}

\newcommand{\Hom}{\operatorname{Hom}}
\newcommand{\Ht}{\operatorname{H}}
\newcommand\id{\operatorname{id}}

\newcommand{\loc}{\operatorname{loc}}
\renewcommand{\mid}{\operatorname{mid}}

\newcommand{\odd}{\operatorname{odd}}

\newcommand{\reg}{ \mathrm{reg} }

\newcommand{\Sh}{\operatorname{Sh}}

\newcommand{\supp}{\operatorname{supp}}

\newcommand\Mand{\text{ and }}

\newcommand\Mfor{\text{ for }}
\newcommand\Mforall{\text{ for all }}

\newcommand\Mif{\text{ if }}

\newcommand\Mor{\text{ or }}

\newcommand\Mst{\text{ s.t. }}

\DeclareMathAlphabet{\mathpzc}{OT1}{pzc}{m}{it}

\newcommand\paperintro%
        {%
         }
\newcommand\paperbody%
        {%
         }

\newcommand\bbB{\mathbb{B}}

\newcommand\bbH{\mathbb{H}}

\newcommand\bbR{\mathbb{R}}

\newcommand\bA{\mathbf{A}}

\newcommand\bE{\mathbf{E}}
\newcommand\bF{\mathbf{F}}

\newcommand\bH{\mathbf{H}}

\newcommand\bL{\mathbf{L}}

\newcommand\bS{\mathbf{S}}

\newcommand\bW{\mathbf{W}}

\newcommand\bOm{\mathbf{\Omega}}

\newcommand\cA{\mathcal{A}}

\newcommand\cC{\mathcal{C}}
\newcommand\cD{\mathcal{D}}
\newcommand\cE{\mathcal{E}}

\newcommand\cH{\mathcal{H}}

\newcommand\cL{\mathcal{L}}

\newcommand\cP{\mathcal{P}}

\newcommand\cU{\mathcal{U}}
\newcommand\cV{\mathcal{V}}
\newcommand\cW{\mathcal{W}}

\newcommand\sA{\mathscr{A}}

\newcommand\sD{\mathscr{D}}

\newcommand\sL{\mathscr{L}}

\datver{\Modified}

\begin{document}
\title{Refined intersection homology on non-Witt spaces}

\author{Pierre Albin}
\address{University of Illinois at Urbana-Champaign}
\email{palbin@illinois.edu}
\author{Markus Banagl}
\address{Mathematisches Institut, Universit\"at Heidelberg}
\email{banagl@mathi.uni-heidelberg.de}
\author{Eric Leichtnam}
\address{CNRS Institut de Math\'ematiques de Jussieu}
\author{Rafe Mazzeo}
\address{Department of Mathematics, Stanford University}
\email{mazzeo@math.stanford.edu}
\author{Paolo Piazza}
\address{Dipartimento di Matematica, Sapienza Universit\`a di Roma}
\email{piazza@mat.uniroma1.it}

\begin{abstract}
We develop a generalization to non-Witt spaces of the intersection homology theory of Goresky-MacPherson. 
The second author has described the self-dual sheaves compatible with intersection homology, and the other authors have described a generalization of Cheeger's $L^2$ de Rham cohomology. In this paper we extend both of these cohomologies by describing all sheaf complexes in the derived category of constructible sheaves that are compatible with middle perversity intersection cohomology, though not necessarily self-dual. 
On Thom-Mather stratified spaces this refined intersection cohomology theory coincides with the analytic de Rham theory.
\end{abstract}

\maketitle

\tableofcontents

The singular homology of a singular space generally does not satisfy Poincar\'e Duality. 
Remarkably, Goresky and MacPherson showed that if the singularities are stratified then there is an adapted homology theory, intersection homology, that satisfies a generalized Poincar\'e Duality \cite{GM}. On `Witt spaces' \cite{Siegel} they defined middle perversity intersection homology groups satisfying Poincar\'e Duality and defining a bordism-invariant signature.
Working independently, Cheeger discovered that the de Rham cohomology of the $L^2$-differential forms with respect to an adapted metric on the regular part of a Witt space also satisfies Poincar\'e Duality \cite{Cheeger:Hodge}. In fact, Cheeger was able to show that his analytic theory is dual to the topological theory of Goresky-MacPherson.

A stratified pseudomanifold is Witt if the lower middle perversity middle-dimensional rational intersection homology of all links of strata of odd codimension vanishes.
If a space is not Witt then there are two middle perversity intersection chain sheaves, one for lower middle perversity and one for upper middle perversity, and the canonical morphism between them is not a
quasi-isomorphism. 
In terms of the sheaf theoretic formulation of intersection homology in \cite{GM2}, this means that none of the Deligne sheaves $\IC p$ is self-dual. A theory of self-dual sheaves on non-Witt spaces, compatible with intersection homology, has been developed by the second author in \cite{BanaglShort}. These sheaves `sit in between' the two middle perversity intersection sheaves and 
provide cohomology theories on non-Witt spaces satisfying Poincar\'e Duality. 
Not every non-Witt space carries a self-dual sheaf compatible with intersection homology, and spaces that do are here termed {\em $L$-spaces}. Strikingly, in \cite{Banagl:LClasses} it is shown that the signature of a self-dual sheaf on an $L$-space is a bordism-invariant and is the same for every choice of self-dual sheaf!

Cheeger's $L^2$ de Rham cohomology theory has been extended to non-Witt spaces by some of the authors, \cite{ALMP13.1}, 
extending \cite{Cheeger:Conic} and \cite{ALMP11}. The failure of the Witt condition means analytically that one must impose
a boundary condition at the singular strata in order that the exterior derivative and the Hodge de Rham operator have closed
domains. These boundary conditions refine the middle perversities of Goresky-MacPherson and were termed 
{\em mezzoperversities} in \cite{ALMP13.1}. Every analytic mezzoperversity has a dual mezzoperversity and the 
intersection pairing induces a generalized Poincar\'e Duality, refining that of intersection homology.
Not every non-Witt space carries a self-dual mezzoperversity, and smoothly stratified spaces that do 
are termed Cheeger spaces in \cite{ALMP13.1}. 

A natural conjecture, which we establish in this paper, is that the $L^2$ cohomology associated to a self-dual mezzoperversity 
and the cohomology of a self-dual sheaf coincide. (Note, however, that as discussed below, the former theory requires a 
Thom-Mather stratification, while the latter only requires a topologically stratified space.) More generally, since the analytic cohomology is defined for mezzoperversities that are not self-dual, we provide an extension of the theory of \cite{BanaglShort} to non-self-dual sheaves that are compatible with intersection homology.

Specifically, we define a category $RP(\hat X)$ associated to a topological pseudomanifold $\hat X$ which captures the candidate sheaves for refining intersection homology on non-Witt spaces. If $\hat X$ is Witt, then $RP(\hat X)$ consists up to isomorphism of a single complex of sheaves
\begin{equation*}
	\hat X \text{ Witt} \implies RP(\hat X) = \{ \IC n \cong \IC m \},
\end{equation*}
while on a general pseudomanifold $RP(\hat X)$ contains $\IC n$ and $\IC m$ and possibly many other sheaf complexes. In particular the category of self-dual sheaves of  \cite{BanaglShort} is a full subcategory of the category of refined middle perversity sheaves,
\begin{equation*}
	SD(\hat X) \subseteq RP(\hat X).
\end{equation*}

Any complex $\bS^\bullet$ in $RP(\hat X)$ will have unique maps
\begin{equation*}
	\IC m \lra \bS^{\bullet} \lra \IC n
\end{equation*}
compatible with the normalizations on the regular part of $\hat X.$
In fact, each $\bS^{\bullet}$ will be shown to be obtained from $\IC m$ by the `addition' of a {\em topological mezzoperversity}, consisting of a compatible choice of locally constant sheaves at each stratum where the Witt condition does not hold.

Given a topological mezzoperversity $\cL,$ we will use a modified Deligne truncation operator $\tau_{\leq \cL}$, which is essentially Banagl's truncation functor
$\tau_{\leq p} (\cdot, \cdot): D(\cdot) \rtimes MS(\cdot)\to D(\cdot)$ constructed in 
\cite{BanaglShort},
to define an object in $RP(\hat X),$
\begin{equation*}
	\mathbf{IC}^\bullet_{\cL} =
	\tau_{\leq \cL(n)} Ri_{n*} \cdots 
	\tau_{\leq \cL(2)} Ri_{2*}\bbR_{\cU_2},
\end{equation*}
and conversely we will show that every object of $RP(\hat X)$ is (up to isomorphism) of this form for a unique topological mezzoperversity $\cL.$

As with the celebrated construction of Deligne-Goresky-MacPherson, we will show that any mezzoperversity gives rise to a \emph{unique} object up to isomorphism in $D(\hat X).$ We will define the `dual topological mezzoperversity' and show that the corresponding homology groups satisfy a global duality. If a topological mezzoperversity is self-dual, then it will be equivalent to a Lagrangian structure in the sense of \cite{BanaglShort} and the resulting object in $D(\hat X)$ is the corresponding self-dual sheaf.

Finally we restrict to Thom-Mather stratified spaces and show that the $L^2$ de Rham cohomology spaces from \cite{ALMP13.1} are canonically isomorphic to refined intersection cohomology groups. Indeed starting from any {\em analytic} mezzoperversity, as we will now refer to the mezzoperversities from \cite{ALMP13.1}, we obtain a sheaf complex in $RP(\hat X)$ by sheafifying the analytic domains associated to the mezzoperversity.
The topological mezzoperversity of this sheaf complex is equivalent to the given analytic mezzoperversity, and we show that every topological mezzoperversity corresponds to an analytic one in this way. 
Since the global sections of the sheafifications turn out to agree with the global analytic domain
of the mezzoperversity,
the $L^2$ de Rham cohomology coincides with refined intersection cohomology.

\medskip
\noindent
{\bf Acknowledgements.} 
P.A. was partly supported by NSF Grant DMS-1104533 and an IHES visiting position and thanks Sapienza 
Universit\`a di Roma, Stanford, and Institut de Math\'ematiques de Jussieu for their hospitality and support.
M.B. was supported in part by a research grant of the Deutsche Forschungsgemeinschaft.
E. L. thanks Sapienza Universit\`a di Roma 
for hospitality during several week-long visits; 
financial support was  provided by CNRS-INDAM 
through the bilateral project  ``Noncommutative Geometry."
R.M. acknowledges support by NSF Grant DMS-1105050. 
P.P. thanks the {\it Projet Alg\`ebres d'Op\'erateurs} of {\it Institut
de Math\'ematiques de Jussieu}
for hospitality during several short visits and a two months long visit
 in the Spring of 2013; financial support was
provided by Universit\'e Paris 7, Instituto Nazionale di Alta Matematica and CNRS-INDAM 
(through the bilateral project  ``Noncommutative Geometry") and Ministero dell'Universit\`a e della Ricerca Scientifica 
(through the project ``Spazi di Moduli e Teoria di Lie"). \\

\smallskip

The authors are grateful to Francesco Bei for many helpful conversations.

\section*{Notation}
Note that this notation is different from that employed in \cite{ALMP11, ALMP13.1}.

Let $\hat X$ be an $n$-dimensional pseudomanifold with a fixed topological stratification
\begin{equation*}
	\emptyset = X_{-1} \subset X_0 \subset \cdots \subset X_{n-2} \subset X_n = \hat X
\end{equation*}
such that each $X_k$ is a closed subset of $\hat X,$ and the $k^{\text{th}}$ stratum
\begin{equation*}
	Y_k = X_k \setminus X_{k-1}
\end{equation*}
is a manifold of dimension $k.$ 
We assume that $Y_n = \hat X \setminus X_{n-2}$ is dense.
We will often work with the increasing sequence of open sets
\begin{equation*}
	\cU_k = \hat X \setminus X_{n-k}
\end{equation*}
for which we have
\begin{equation*}
	\cU_{k+1} = \cU_k \cup Y_{n-k},
\end{equation*}
and we make frequent use of the inclusions of these two subsets: 
\begin{equation*}
	\xymatrix{
	\cU_k \ar@{^(->}[r]^{i_k} & \cU_{k+1}  \ar@{<-^)}[r]^{j_k} & Y_{n-k}.
	}
\end{equation*}
The singular and regular parts of $\hat X$ are $X_{n-2}$ and $\cU_2$, respectively.
Each point $x \in X_{k}\setminus X_{k-1}$ has distinguished neighborhoods in $X$ homeomorphic to the product of a Euclidean ball of dimension $k$ and an open cone 
\begin{equation*}
	\bbB^k \times C^\circ(Z)
\end{equation*}
for some compact stratified space $Z$ of dimension $n-k-1,$ called the link of $Y_k$ at $x.$
Moreover the stratifications of $Z$ and $\hat X$ are compatible through this homeomorphism.

We denote the category of sheaves of $\bbR$-vector spaces on $\hat X$ by $\Sh(\hat X),$ the category of constructible differential graded sheaves (i.e., complexes $\sc A$ such that each $\sc A|_{Y_k}$ has locally constant cohomology sheaves and has finitely generated stalk cohomology) by $\Sh^{\bullet}(\hat X),$ and the derived category of bounded constructible complexes by $D(\hat X).$ We will denote the derived sheaf of an object $\sc A$ in $D(\hat X),$ i.e., the sheaf associated to the presheaf $( \cU \mapsto \Ht^\bullet(\sc A(\cU)) ),$ by
\begin{equation*}
	\bH^\bullet(\sc A) \in \Sh^{\bullet}(\hat X)
\end{equation*}
and its global hypercohomology, i.e., the cohomology of the global sections of an injective resolution, by $\bbH^\bullet(\hat X; \sc A).$ 
We denote the intersection sheaf complex associated to the lower middle perversity $\bar m$ of Goresky-MacPherson by $\IC m,$ and similarly for the upper middle perversity $\bar n$ by $\IC n.$ For a review of these concepts we refer to \cite{GM2, Borel, BanaglLong}.

When we study differential forms, we will assume that the stratification is `smooth' as described below.
In the introduction we distinguish between {\em topological} mezzoperversities and {\em analytic} mezzoperversities. In the text we will reserve this distinction until the final section, before that we develop the notion of topological mezzoperversity and refer to it as simply a mezzoperversity.

%%%%%%%%%%%%%%%%%%%%%%%%%%%
\section{Refined middle perversity sheaves}
%%%%%%%%%%%%%%%%%%%%%%%%%%%

We define the objects of interest in the derived category, following \cite{GM2, BanaglShort}.
In this section, we work with topologically stratified pseudomanifolds and sheaves of $\bbR$-vector spaces.
Let $i_x: \{ x \} \hookrightarrow \hat X$ denote the inclusion of a point $x\in \hat X$.

\begin{definition}
Let $\hat X$ be a stratified orientable pseudomanifold and $\sc S$ a constructible bounded complex of sheaves. We say that $\sc S$ satisfies the axioms $[RP],$ or is a {\bf refined middle perversity complex of sheaves}, provided:\\
\begin{itemize}
\item [\textup{(}{\bf RP1}\textup{)}]
Normalization: There is an isomorphism of the restriction of $\sc S$ to the regular part $\cU_2$ of $\hat X$ and the constant rank $1$ sheaf over $\cU_2,$
\begin{equation*}
	\nu^{\bS}:\bbR_{\cU_2} \xlra{\cong} \sc S\rest{\cU_2}
\end{equation*}
\item [\textup{(}{\bf RP2}\textup{)}]
Lower bound: $\Ht^\ell(i_x^*\sc S) =0$ for any $x\in \hat X$ and $\ell < 0.$
\item [\textup{(}{\bf RP3}\textup{)}]
$\bar n$-stalk vanishing: $\Ht^\ell(i_x^*\sc S) =0$ for  any $x\in \cU_{k+1}\setminus \cU_2$ and $\ell > \bar n(k).$
\item [\textup{(}{\bf RP4}\textup{)}]
$\bar m$-costalk vanishing: $\Ht^\ell(i_x^!\sc S) =0$ for  any $x\in Y_{n-k}$ and $\ell \leq \bar m(k)+n-k+1.$
\end{itemize}
We denote by $RP(\hat X)$ the full subcategory of $D(\hat X)$ whose objects satisfy the axioms $[RP].$
\end{definition}

\begin{remark}
For simplicity we work with orientable $\hat X$ and the constant sheaf over the regular part of $\hat X.$ With essentially no change one can allow a locally constant system over general $\hat X.$ For example, the canonical normalization over a non-orientable $\hat X$ is to use the orientation sheaf over $\cU_2.$
\end{remark}

\begin{remark}
The category $RP(\hat X)$ is closely related to the category $EP(\hat X)$ of equiperverse sheaves defined in \cite[\S 9.3.1]{BanaglLong}.
Indeed, the stalk and costalk vanishing conditions demanded in the respective definitions agree on strata of odd codimension.
For $EP,$ cohomology stalks on a stratum of even codimension $k$ must vanish above $\bar n(k) +1,$ while for $RP$ they already have to vanish above $\bar n(k).$
Similarly, there is a discrepancy of one degree for the costalk vanishing conditions.
Hence $RP(\hat X)$ is always a full subcategory of $EP(\hat X),$ and they coincide if $\hat X$ has only singular strata of odd codimension.
If $\hat X$ has only strata of even codimension, then $EP(\hat X)$ equals ${}^{\bar m}\cP(X),$ the category of middle perverse sheaves on $\hat X,$ while $RP(\hat X)$ contains up to isomorphism only $\IC m.$
\end{remark}

As in \cite{GM2}, the axiom (RP4) is equivalent to knowing that, at each $x \in \cU_{k+1} \setminus \cU_k,$
the attaching maps 
\begin{equation}\label{eq:AttachingAxiom}
	\bH^j(\sc S_x) \lra \bH^j(R i_{k*}i_k^* \sc S)_x
\end{equation}
are isomorphisms for all $j \leq \bar m(k).$ \\

\begin{remark}
One could also define refined sheaves for other perversities, but as our examples correspond to middle perversity, we do not explore this here.
\end{remark}

Let us immediately note two important properties of the axioms.

\begin{proposition} \label{prop:FirstProp}$ $
\begin{itemize}
\item [i)] If $\sc S \in D(\hat X)$ satisfies axioms [RP], and $\cD\sc S$ is its Verdier dual, then $(\cD\sc S) [-n]$ satisfies the axioms [RP].
\item [ii)] {[Cappell-Shaneson]} If $\sc A$ and $\sc B$ are in $RP(\hat X),$ the restriction to $\cU_2$ induces an injective map
\begin{equation*}
	\Hom_{D(\hat X)}(\sc A, \sc B) \lra 
	\Hom_{D(\cU_2)}(\sc A\rest{\cU_2}, \sc B\rest{\cU_2}).
\end{equation*}
\end{itemize}

\end{proposition}
\begin{proof}
These follow directly from the axioms:
For constructible complexes of sheaves satisfying ($RP1$) and ($RP2$), it is well-known \cite[\S5.3]{GM2} that Verdier duality interchanges conditions ($RP3$) and ($RP4$), as $\bar{m}$ and $\bar{n}$ are
complementary perversities.
In \cite[(1.3)]{CS91}, ($ii$) is shown to follow from ($RP3$) and ($RP4$).
\end{proof}

\begin{example}
Let $\hat X=M$ be an $n$-dimensional oriented manifold (without boundary).
Let $\sc S = \bbR_M$ be the constant real sheaf in degree $0$ on $M$.
Then $\sc S$ is an object of $RP (M)$. (Axioms (RP3) and (RP4) are vacuously satisfied.)
If $\cA$ is any local coefficient system (locally constant sheaf) in degree $0$ on an oriented manifold $M$, then
its Verdier dual is given by $\cD \cA = \cA^* [n]$, where $\cA^*$ is the linear dual local system
with stalks $\cA^*_x = \Hom (\cA_x, \bbR),$ $x\in M$. In particular, the Verdier dual of a local
system on a manifold is again a local system, but there is a degree shift.
Applying this to $\sc S$, we obtain the self-duality isomorphism
\[ (\cD \sc S)[-n] = \bbR^*_M \cong \bbR_M = \sc S, \]
where the isomorphism $\bbR^*_M \cong \bbR_M$ is given by the canonical
multiplication $\bbR_M \otimes \bbR_M \to \bbR_M$.
This shows in particular that $\cD \sc S [-n]$ is again an object of $RP(M)$.
\end{example}

Let us recall Deligne's construction of the sheaf complexes $\IC p$ in $D(\hat X).$
For any perversity $\bar p$ on $\hat X$ and $k= 2,\ldots, n$ define
\begin{equation*}
	E^{\bar p}_k:D(\cU_k) \lra D(\cU_{k+1}), \quad
	E^{\bar p}_k \sc A = \tau_{\leq \bar p(k)}Ri_{k*}\sc A,
\end{equation*}
and then set
\begin{equation*}
	\IC p = E^{\bar p}_n\cdots E^{\bar p}_2 \bbR_{\cU_2}.
\end{equation*}
The sheaf complexes $\IC m$ and $\IC n$ both satisfy the axioms [RP], and we will now show that any sheaf that satisfies [RP] is closely related to these sheaves.

First, on the regular part $\cU_2$ of $\hat X,$ any sheaf $\sc S \in RP(\hat X)$ has natural maps
\begin{equation}\label{eq:TrivialFirstIso}
	\IC m \rest{\cU_2} \lra 
	\sc S \rest{\cU_2} \lra
	\IC n \rest{\cU_2}
\end{equation}
since by (RP1) each of these complexes has an isomorphism to the constant sheaf $\bbR_{\cU_2}.$

\begin{proposition}[{cf. \cite[\S 2.2]{BanaglShort}}] \label{prop:UniquePair}
For any $\sc S \in RP(\hat X)$ there is a unique pair of morphisms
\begin{equation*}
	\IC m \xlra\alpha \sc S \xlra\beta \IC n
\end{equation*}
extending \eqref{eq:TrivialFirstIso}. \\

Moreover, on $\cU_{k+1} = \hat X \setminus X_{n-k-1},$ these morphisms factor
\begin{equation}\label{eq:FactorAlphaBeta}
 	\xymatrix{ 
	i_{k+1}^*\IC m \ar[rr]^{\alpha\rest{\cU_{k+1}}} \ar[rd]_{a'} & &
	i_{k+1}^*\sc S \ar[rr]^{\beta\rest{\cU_{k+1}}} \ar[rd]_{b} & &
	i_{k+1}^*\IC n \\
	& E^{\bar m}_k i_k^* \sc S \ar[ru]_a &
	& E^{\bar n}_k i_k^* \sc S \ar[ru]_{b'} &}
\end{equation}
through two standard extensions $E^{\bar m}_k i_k^* \sc S$ and $E^{\bar n}_k i_k^* \sc S$ of $i_k^* \sc S.$
\label{Prop:ExistUniqueMap}\end{proposition}

\begin{proof} 
Uniqueness follows from Proposition \ref{prop:FirstProp}($ii$), so we only need to establish existence.

Assume inductively that we have found morphisms
\begin{equation}\label{eq:TrivialInductiveIso}
	i_k^*\IC m \lra
	i_k^*\sc S \lra
	i_k^*\IC n
\end{equation}
extending \eqref{eq:TrivialFirstIso}.
Applying the functors $E^{\bar p}_k$ to \eqref{eq:TrivialInductiveIso} we obtain morphisms
\begin{equation*}
	a': i_{k+1}^*\IC m 
	= E^{\bar m}_k( i_k^*\IC m)
	\lra E^{\bar m}_k( i_k^*\sc S  ), \quad
	b':E^{\bar n}_k( i_k^*\sc S ) \lra
	E^{\bar n}_k( i_k^*\IC n ) 
	=i_{k+1}^*\IC n 
\end{equation*}
in $D(\cU_{k+1}).$
Adjunction $i_{k+1}^*\sc S \lra Ri_{k*}i_k^*\sc S$ is the identity on $\cU_k$ and induces maps
\begin{equation*}
	\tau_{\leq \bar p(k)} i_{k+1}^*\sc S
	\lra E^{\bar p}_k( i_k^*\sc S )
	\Mfor \bar p = \bar m \Mor \bar n.
\end{equation*}
The axiom (RP4) (in the form \eqref{eq:AttachingAxiom}) implies that for $\bar m$ this morphism is a quasi-isomorphism, and axiom (RP3) implies that 
\begin{equation*}
	i_{k+1}^*\sc S \cong \tau_{\leq \bar n(k)} i_{k+1}^*\sc S
\end{equation*}
and so all together we have morphisms
\begin{equation*}
	i_{k+1}^*\IC m
	\xlra{a'} 
	E^{\bar m}_k( i_k^*\sc S ) \cong \tau_{\leq \bar m(k)} i_{k+1}^*\sc S
	\xlra a 
	i_{k+1}^*\sc S \cong \tau_{\leq \bar n(k)} i_{k+1}^*\sc S
	\xlra b 
	E^{\bar n}_k( i_k^*\sc S ) \xlra{b'}
	i_{k+1}^*\IC n
\end{equation*}
extending \eqref{eq:TrivialFirstIso} as required.
\end{proof}

Let $\sc C(a),$ $\sc C(b)$ and $\sc C(ba)$ be sheaf complexes completing $a,b$ and $ba$, respectively, to distinguished triangles. Using these triangles we are able to conclude the following:

\begin{proposition}
For any $x \in \cU_{k+1}\setminus \cU_k,$ the cohomology of the complex $\sc C(ba)_x$ is concentrated in degree $\bar n(k),$ and the maps $a$ and $b$ in \eqref{eq:FactorAlphaBeta} induce an injective map
\begin{equation*}
	\bH^{\bar n(k)}(\sc S_x)  \lra \bH^{\bar n(k)}(\sc C(ba)_x).
\end{equation*}
\end{proposition}

\begin{proof}
Let us analyze $a$ and $b$ for odd $k$. The distinguished triangles
\begin{equation*}
\begin{gathered}
	\xymatrix{ 
	E^{\bar m}_k i_k^*\sc S \ar[rr]^a & & i_{k+1}^*\sc S \ar[ld] \\
	& \sc C(a) \ar[lu]^{+1} & }, \quad
	\xymatrix{ 
	i_{k+1}^*\sc S \ar[rr]^b & & E^{\bar n}_k i_{k}^*\sc S \ar[ld] \\
	& \sc C(b) \ar[lu]^{+1} & }, \\
	\xymatrix{ 
	E^{\bar m}_k i_k^*\sc S \ar[rr]^{ba} & & E^{\bar n}_k i_{k}^*\sc S \ar[ld] \\
	& \sc C(ba) \ar[lu]^{+1} & }
\end{gathered}
\end{equation*}
fit together into a diagram
\begin{equation*}
\begin{gathered}
	\phantom{xxx}\\
	\xymatrix{
	E^{\bar m}_k i_k^*\sc S  \ar@/^25pt/[rrrr]^{ba} \ar[rr]^{a} & & i_{k+1}^*\sc S \ar[ld] \ar[rr]^{b}
		& & E^{\bar n}_k i_{k}^*\sc S  \ar@/^30pt/[ddll] \ar[ld] \\
	& \sc C(a) \ar[lu]^{+1}  & & \sc C(b)  \ar[lu]^{+1} & \\
	& & \sc C(ba)  \ar@/^30pt/[lluu]^{+1} & & }
\end{gathered}
\end{equation*}
which by the `octahedral axiom' we can complete to the octahedral diagram
\begin{equation}\label{eq:FunOctahedron}
\begin{gathered}
	\phantom{xxx}\\
	\xymatrix{
	E^{\bar m}_k i_k^*\sc S  \ar@/^25pt/[rrrr]^{ba} \ar[rr]^{a} & & i_{k+1}^*\sc S \ar[ld] \ar[rr]^{b}
		& & E^{\bar n}_k i_{k}^*\sc S  \ar@/^30pt/[ddll] \ar[ld] \\
	& \sc C(a) \ar[lu]^{+1} \ar[rd] & & \sc C(b) \ar[ll]^{+1} \ar[lu]^{+1} & \\
	& & \sc C(ba) \ar[ru] \ar@/^30pt/[lluu]^{+1} & & }
\end{gathered}
\end{equation}
Taking the stalks of the cohomology sheaves at a point $x \in \cU_{k+1}\setminus \cU_k,$ we see that
\begin{equation*}
\begin{gathered}
	\phantom{xxx}\\
	\xymatrix{
	\bH^j( E^{\bar m}_k i_k^*\sc S )_x  \ar@/^25pt/[rrrr]^{\bH^j(ba)} \ar[rr]^{\bH^j(a)} 
		& & \bH^j(\sc S_x) \ar[ld] \ar[rr]^{\bH^j(b)}
		& & \bH^j( E^{\bar n}_k i_k^*\sc S )_x  \ar@/^30pt/[ddll] \ar[ld] \\
	& \bH^j(\sc C(a)_x) \ar@{-->}[lu]^{+1} \ar[rd] & & 
		\bH^j(\sc C(b)_x) \ar@{-->}[ll]^{+1} \ar@{-->}[lu]^{+1} & \\
	& & \bH^j(\sc C(ba)_x) \ar[ru] \ar@/^30pt/@{-->}[lluu]^{+1} & & }
\end{gathered}
\end{equation*}
where we have used dotted arrows for maps into the corresponding $\bH^{j+1}.$
From \eqref{eq:AttachingAxiom}
we know that the map $\bH^j(b)$ is an isomorphism for $j \leq \bar m(k),$
and by construction $\bH^j(ba)$ is the identity map for these $j,$ hence $\bH^j(a)$ is also an isomorphism and
\begin{equation*}
\begin{gathered}
	\bH^j(\sc C(ba)_x) = \bH^j(\sc C(a)_x) =0 \Mforall j \leq \bar m(k), \\
	 \bH^j(\sc C(b)_x)=0 \Mforall j < \bar m(k).
\end{gathered}
\end{equation*}
If $j>\bar n(k)$, then $\bH^j( E^{\bar m}_k i_k^*\sc S )_x$ and $ \bH^j(\sc S_x)$
both vanish and thus $\bH^j(\sc C(a)_x)=0$ for these $j$. We conclude that $\sc C(a)$
is concentrated in a single degree, namely $\bar n(k)$. Moreover, if $x\in \cU_k$, then
$\bH^j(\sc C(a)_x)=0$ for \emph{all} $j$, since $i^*_k a$ is the identity map. So $\sc C(a)$ is supported over the manifold $\cU_{k+1} \setminus \cU_k$. 
Summarizing, $\sc C(a)$ has the form
\[ \sc C(a) = j_{k*} \cA [-\bar n(k)], \]
where $\cA$ is a locally constant sheaf on the manifold $M=Y_{n-k}$.
Note that as $j_k$ is a closed inclusion, the functor $j_{k*}$ is just extension by zero and
we have $j_{k*} = j_{k!}$.

We shall now show that $\sc C(b)$ is also concentrated in the degree $\bar n(k)$.
We prove this following \cite[\S 2.4]{BanaglShort}.
As pointed out in Proposition \ref{prop:FirstProp}($i$), the Verdier dual $\cD(\sc S)[-n]$ is also in $RP(\hat X).$ Thus we have a distinguished triangle
\begin{equation*}
	\xymatrix{ 
	i_{k+1}^*\cD(\sc S)[-n] \ar[rr]^{b'} & & E^{\bar n}_k i_{k}^*\cD(\sc S)[-n] \ar[ld] \\
	& \sc C(b') \ar[lu]^{+1} & }.
\end{equation*}
which dualizes to
\begin{equation*}
	\xymatrix{ 
	E^{\bar m}_k i_k^*\sc S \ar[rr]^{\cD(b')[-n]} & & i_{k+1}^*\sc S \ar[ld] \\
	& \cD(\sc C(b'))[-n+1] \ar[lu]^{+1} & }.
\end{equation*}
By Proposition \ref{prop:FirstProp}($ii$), we know that $\cD(b')[-n] = a$ and so we have an isomorphism of triangles
\begin{equation*}
	\xymatrix{
	E^{\bar m}_k i_k^*\sc S \ar[r]^{\cD(b')[-n]} \ar[d]^=&  i_{k+1}^*\sc S \ar[r] \ar[d]^= & 
    \cD(\sc C(b'))[-n+1] \ar[r]^-{+1} \ar@{-->}[d]^{\cong} & \ldots \\
	E^{\bar m}_k i_k^*\sc S \ar[r]^{a} &  i_{k+1}^*\sc S \ar[r]& \sc C(a) \ar[r]^-{+1} & \ldots }
\end{equation*}
which shows that $\cD(\sc C(b'))[-n+1] \cong \sc C(a).$ 
In the following calculation, we distinguish in our notation between the Verdier dualizing
functor $\cD_{\hat X} =\cD$ on $\hat X$ and the dualizing functor $\cD_M$ on the manifold
$M = Y_{n-k}$. We have
\begin{eqnarray*}
\sc C(b') & \cong &
 (\cD_{\hat X} \sc C(a))[-n+1] \\
& \cong &  (\cD_{\hat X} ( j_{k*} \cA [-\bar n(k)]))[-n+1] \\
& \cong &  (\cD_{\hat X} ( j_{k*} \cA))[\bar n(k)-n+1] \\
& \cong &  j_{k!} (\cD_{M} \cA)[\bar n(k)-n+1] \\
& \cong &  j_{k*} \cA^* [n-k][\bar n(k)-n+1] \\
& \cong &  j_{k*} \cA^* [\bar n(k)-k+1] \\
& \cong &  j_{k*} \cA^* [-\bar n(k)]. \\
\end{eqnarray*}
(Recall that for odd $k$, $\bar n(k) = (k-1)/2$.)
This shows that $\sc C(b')$ is concentrated in degree $\bar n(k)$.
Since every object of $RP(\hat X)$ is of the form $\cD \sc S [-n]$ for some
$\sc S \in \operatorname{Ob}RP(\hat X),$ this also shows that $\sc C(b)$ is concentrated
in degree $\bar n(k)$.

If $k$ is odd, then $\bar n(k) = \bar m(k)+1$ and, using the information above, the diagram for $\bar n(k)$ collapses to a pair of short exact sequences
\begin{equation*}
	\xymatrix{
	0 \ar[r] & \bH^{\bar n(k)}(\sc S_x) \ar[d]^-{\cong} \ar[r]^-{\bH^{\bar n(k)}(b)} &
	\bH^{\bar n(k)}(E^{\bar n}_ki_k^*\sc S)_x \ar[d]^-{\cong} \ar[r] &
	\bH^{\bar n(k)}( \sc C(b)_x) \ar[d]^-{=} \ar[r] & 0 \\
	0 \ar[r] & \bH^{\bar n(k)}(\sc C(a)_x) \ar[r] &
	\bH^{\bar n(k)}(\sc C(ba)_x)  \ar[r] &
	\bH^{\bar n(k)}( \sc C(b)_x) \ar[r] & 0 }
\end{equation*}
and so we have in particular an injective map $\bH^{\bar n(k)}(\sc S_x) \cong \bH^{\bar n(k)}(\sc C(a)_x) \lra \bH^{\bar n(k)}(\sc C(ba)_x).$

\end{proof}

One interpretation of this proposition is that $i_{k+1}^*\sc S$ differs from the two canonical extensions of $i_k^*\sc S$ to $\cU_{k+1},$ namely $E^{\bar m}_k i_k^* \sc S$ and $E^{\bar n}_k i_k^* \sc S,$ by the choice of a subsheaf $\bH^{\bar n(k)}(\sc S_x)$ of $\bH^{\bar n(k)}(\sc C(ba)_x).$
Indeed, it is possible to assemble $i^*_k\sc S$ and the sheaf
\begin{equation}\label{eq:bWYk}
	\bW_{\bS}(Y_{n-k}) \in \Sh(Y_{n-k}), \quad
	\bW_{\bS}(Y_{n-k}) = \bH^{\bar n(k)}(\sc S)\rest{Y_{n-k}}. 
\end{equation}
to an object over $\cU_{k+1}$ satisfying $[RP],$ quasi-isomorphic to $\sc S|_{\cU_{k+1}}.$

To carry out this assembly, let us recall the 
 modified truncation functor from \cite[\S 5]{BanaglShort}:
Given a constructible complex of sheaves $\sc A$ on a stratified pseudomanifold $M,$ and an injective sheaf map
\begin{equation*}
	\phi: \bE \lra \bH^p(\sc A),
\end{equation*}
we use the quotient map
\begin{equation*}
	\pi: \text{ {\bf ker} } d^p \lra \bH^p(\sc A)
\end{equation*}
to define a new (constructible) complex of sheaves
\begin{equation*}
	\tau_{\leq p}(\sc A, \bE) \in \Sh^\bullet(M), \quad
	\lrpar{ \tau_{\leq p}(\sc A, \bE) }^j
	= \begin{cases}
	\bA^j & \Mif j < p \\
	\pi^{-1}(\phi(\bE)) & \Mif j=p \\
	0 & \Mif j>p
	\end{cases} 
\end{equation*}
Notice that this new complex satisfies
\begin{equation}\label{eq:CohoTau}
	\bH^j ( \tau_{\leq p}(\sc A, \bE) )
	= \begin{cases}
	\bH^j(\sc A) & \Mif j < p \\
	\phi(\bE) & \Mif j=p \\
	0 & \Mif j>p
	\end{cases} 	
\end{equation}

It is an important fact that this modified truncation defines a functor.
Let $\Sh^\bullet(M)$ denote the category of complexes of sheaves on $M,$ and $MS(M)$ the associated category of morphisms.
Let $\Sh^{\bullet}(M) \rtimes MS(M)$ denote the twisted product category whose objects are pairs
\begin{equation*}
	(\sc A, \bE \xlra\phi \bH^p(\sc A) )
\end{equation*}
with $\phi$ injective, and whose morphisms are pairs
\begin{equation*}
	(f,h) \in \Hom \lrpar{
	(\sc A, \bE \xlra\phi \bH^p(\sc A) ),
	(\sc B, \bF \xlra\psi \bH^p(\sc B) ) }
\end{equation*}
with $f: \sc A \lra \sc B$ a sheaf complex morphism and $h: \bE \lra \bF$ a sheaf morphism such that
\begin{equation*}
	\xymatrix{
	\bE \ar[r]^-{\phi} \ar[d]^-h & \bH^p(\sc A) \ar[d]^-{\bH^p(f)} \\
	\bF \ar[r]^-{\psi} & \bH^p(\sc B) }
\end{equation*}
commutes.
In a similar fashion, we define $D(M) \rtimes MS(M)$ starting from the derived category of bounded constructible sheaf complexes on $M.$ We know, from \cite[Theorem 2.5]{BanaglShort}, that modified truncation defines a covariant functor
\begin{equation*}
	\tau_{\leq p}(\cdot, \cdot) : 
	\Sh^{\bullet}(M) \rtimes MS(M) \lra \Sh^{\bullet}(M)
\end{equation*}
which descends to the derived category
\begin{equation*}
	\tau_{\leq p}(\cdot, \cdot) : 
	D(M) \rtimes MS(M) \lra D(M).
\end{equation*}
We can now carry out the assembly referred to above.

\begin{proposition}[{cf. \cite[Lemma 2.4]{Borel}}] \label{prop:SIsItsExtension}
Let $\sc S \in D(\cU_{k+1})$ satisfy the axioms $[RP]$ on $\cU_{k+1},$ let $\bW_{\bS}(Y_{n-k})$ be the sheaf from \eqref{eq:bWYk}, and let
\begin{equation*}
	Ei_k^*\sc S = 
	\begin{cases}
	\tau_{\leq \bar n(k)} Ri_{k*}i_k^*\sc S & \Mif k \even \\
	\tau_{\leq \bar n(k)} (Ri_{k*}i_k^*\sc S, \bW_{\bS}(Y_{n-k})) & \Mif k \odd \\
	\end{cases}
\end{equation*}
then $\sc S \cong Ei_k^*\sc S.$
\end{proposition}

\begin{proof}
We must distinguish two cases: The case $k$ even and the case $k$ odd.
Suppose first that $k$ is even.
By axiom (RP4), phrased as \eqref{eq:AttachingAxiom}, the adjunction map $\sc S \xlra{f} Ri_{k*}i_k^*\sc S$ induces an isomorphism
\[ \tau_{\leq \bar n(k)} \sc S \xlra{f}  \tau_{\leq \bar n(k)} Ri_{k*}i_k^*\sc S, \]
since we have $\bar n(k) = \bar m(k)$ when $k$ is even.
The canonical map $\tau_{\leq \bar n(k)} \sc S \to \sc S$ is an isomorphism by axiom (RP3).
Putting these two isomorphism together, we obtain an isomorphism
$\sc S \cong  \tau_{\leq \bar n(k)} Ri_{k*}i_k^*\sc S = Ei_k^*\sc S$ as claimed.\\

Suppose that $k$ is odd.
The adjunction map $f$ participates in a commutative diagram
\begin{equation*}
	\xymatrix{
	\bW_{\bS}(Y_{n-k}) \ar[rr]^{\id} \ar[d]^{\id} && \bH^{\bar n(k)}(\sc S) \ar[d]^{\bH^{\bar n(k)}(f)} \\
	\bH^{\bar n(k)}(\sc S) \ar[rr]_-{\bH^{\bar n(k)}(f)} & & \bH^{\bar n(k)}(Ri_{k*}i_k^*\sc S) }
\end{equation*}
which by functoriality of 
$\tau_{\leq \bar n(k)}(\cdot, \cdot)$
induces a map 
\[ \tau_{\leq \bar n(k)}(\sc S, \bW_{\bS}(Y_{n-k})) \longrightarrow Ei_k^*\sc S. \]
This map
is an isomorphism: In degrees $\leq \bar m(k)$ it is a cohomology isomorphism by ($RP4$) in the form \eqref{eq:AttachingAxiom}; in degree $\bar n(k)$ it induces the map
\[ \bH^{\bar n(k)}(f):  \bH^{\bar n(k)}(\sc S) \to \bH^{\bar n(k)}(f) (\bH^{\bar n(k)}(\sc S)), \]
also an isomorphism. The canonical map
\begin{equation*}
	\tau_{\leq \bar n(k)}(\sc S,\bW_{\bS}(Y_{n-k})) =  \tau_{\leq \bar n(k)} \sc S \lra \sc S 
\end{equation*}
is an isomorphism by ($RP3$).
Consequently,
\[ \sc S \cong \tau_{\leq \bar n(k)}(\sc S,\bW_{\bS}(Y_{n-k})) \cong Ei_k^*\sc S, \]
as required.
\end{proof}

%%%%%%%%%%%%%%%%%%%%%%%%%%%
\section{Mezzoperversities}
%%%%%%%%%%%%%%%%%%%%%%%%%%%

We start by describing the analogue of Deligne's construction in the category [RP].
Recall from \cite{GM2} that given a perversity, one can inductively construct the intersection sheaf complex in $D(\cU_k)$ for each $k,$ obtaining finally an object in $D(\hat X).$
To construct an object in [RP] one needs more information than just a classical perversity. We examine the analogous procedure for constructing an object in [RP] inductively over $\cU_k.$

All sheaves in [RP] are isomorphic to the constant sheaf $\cP_2 = \bbR$ over $\cU_2$ and to $\cP_3 = \tau_{\leq \bar n(2)}Ri_{2*}\bbR$ over $\cU_3.$
If $\bW(Y_{n-3})$ is a subsheaf of $\bH^{\bar n(3)}(Ri_{3*}\cP_3)$ over $Y_{n-3}$ then let us set
\begin{equation*}
	\cL_4 = \{ \bW(Y_{n-3}) \}, \quad
	\cP_4(\cL_4) 
	= \tau_{\leq \bar n(3)} (Ri_{3*}\cP_3, \bW(Y_{n-3})) \text{ over } \cU_4.
\end{equation*}
We can continue inductively in this way to construct an element $\cP(\cL)$ of $D(\hat X).$
Indeed, if we have constructed $\cL_k$ and $\cP_k(\cL_k)$ over $\cU_k,$ then we extend our construction to $\cU_{k+1}$ by either
\begin{equation*}
	\cL_{k+1} = \cL_k = \{ \bW(Y_{n-3}), \ldots, \bW(Y_{n-(k-1)}) \}, \quad
	\cP_{k+1}(\cL_{k+1}) = \tau_{\leq \bar n(k)}Ri_{k*} \cP_k(\cL_k)
\end{equation*}
if $k$ is even or, if $k$ is odd, choosing 
\begin{equation*}
	\bW(Y_{n-k}) \text{ a subsheaf of } \bH^{\bar n(k)}(Ri_{k*} \cP_k(\cL_k) )
\end{equation*}
and then setting
\begin{equation*}
\begin{gathered}
	\cL_{k+1} = \cL_k \cup \{ \bW(Y_{n-k}) \} = \{ \bW(Y_{n-3}), \ldots, \bW(Y_{n-k}) \}, \\
	\cP_{k+1}(\cL_{k+1}) = \tau_{\leq \bar n(k)}(Ri_{k*} \cP_k(\cL_k), \bW(Y_{n-k}) ).
\end{gathered}
\end{equation*}

\begin{definition}
We refer to the sequence of sheaves
\begin{equation*}
	\bW(Y_{n-k}) \in \Sh(Y_{n-k})
\end{equation*}
together with the injective sheaf maps
\begin{equation*}
	\bW(Y_{n-k}) \lra 
	\bH^{\bar n(k)}(Ri_{k*} \cP_k(\cL_k) )
\end{equation*}
as a (topological) {\bf mezzoperversity} $\cL$ and to the resulting sheaf complex $\cP(\cL)$
as the {\bf Deligne sheaf} associated to the mezzoperversity.
\end{definition}

As anticipated, every $\cP(\cL)$ is an element of $[RP].$
Indeed, it clearly satisfies (RP1)-(RP3), and from \eqref{eq:CohoTau}
we know that the map
\begin{equation*}
	\cP_k (\cL) \lra Ri_{k*}\cP_{k-1}(\cL)
\end{equation*}
is a quasi-isomorphism for degrees $\leq \bar m(k).$ The commutative diagram
\begin{equation*}
	\xymatrix{
	\cP_k (\cL) \ar[rr]^-{\text{adj.}} \ar[rd]_{\text{q. iso }\leq \bar m(k) } & & Ri_{k*}i_k^*\cP_{k}(\cL) = Ri_{k*}\cP_{k-1}(\cL) \ar[ld]^{\cong} \\
	& Ri_{k*}\cP_{k-1}(\cL) & }
\end{equation*}
where we have used that $i_k^*\cP_k (\cL) = \cP_{k-1}(\cL),$
shows that the attaching maps
\begin{equation*}
	\bH^j(\cP_k (\cL)_x) \lra \bH^j (R i_{k*}i_k^* \cP_k(\cL))_x
\end{equation*}
are isomorphisms for all $j \leq \bar m(k).$ We next prove that conversely every element of [RP] is isomorphic to a Deligne sheaf.

\begin{theorem}\label{thm:DeligneSheaf}
If $\sc S$ is an object in $RP(\hat X)$ and $\nu_2: \bbR_{\cU_2} \xlra{\cong} \sc S\rest{\cU_2}$ is the normalization isomorphism from (RP1), then there is a mezzoperversity $\cL$ with Deligne sheaf $\cP(\cL)$ and an isomorphism 
\begin{equation*}
	\nu: \cP(\cL) \lra \sc S
\end{equation*}
extending $\nu_2$ to $\hat X.$
\end{theorem}

\begin{proof}
Over $\cU_2$ there is nothing to show.\\
Over $\cU_3$ we have the isomorphism
\begin{equation*}
	\nu_3 = (\tau_{\leq \bar n(2)}Ri_{2*})(\nu_2): 
	\tau_{\leq \bar n(2)}Ri_{2*}\cP_2 = \cP_3 \lra
	\tau_{\leq \bar n(2)}Ri_{2*} \sc S\rest{\cU_2} \cong \sc S\rest{\cU_3},
\end{equation*}
so the statement holds over $\cU_3.$ 
Over $\cU_4$ we know from Proposition \ref{prop:SIsItsExtension} that
\begin{equation*}
	\sc S\rest{\cU_4} \cong
	\tau_{\leq \bar n(3)} (Ri_{3*}i_3^*\sc S, \bW_{\bS}(Y_{n-3}))
\end{equation*}
where $\bW_{\bS}(Y_{n-3}) = \bH^{\bar n(3)}(\sc S)\rest{Y_{n-3}}.$
Since the map
\begin{equation*}
	\bH^{\bar n(3)}(Ri_{3*}\cP_3)
	\xlra{ \bH^{\bar n(3)}(Ri_{3*} \nu_3) }
	\bH^{\bar n(3)}(Ri_{3*}\sc S\rest{\cU_3})
\end{equation*}
is an isomorphism, we can define an injective sheaf map $\psi$ by the diagram
\begin{equation*}
	\xymatrix{
	\bW_{\bS}(Y_{n-3}) \ar[r]^-{\phi} \ar[d]_{\id} & \bH^{\bar n(3)}(Ri_{3*}\sc S\rest{\cU_3})\\
	\bW_{\bS}(Y_{n-3}) \ar[r]^-{\psi}  & \bH^{\bar n(3)}(Ri_{3*}\cP_3) 
   \ar[u]^{\cong}_{\bH^{\bar n(3)}(Ri_{3*} \nu_3) } }
\end{equation*}
Here, $\phi$ is the map induced on $\bH^{\bar n(3)}$ by the adjunction
$\sc S|_{\cU_4} \to Ri_{3*} i^*_3 \sc S|_{\cU_4}$.
This diagram shows that the pair $(Ri_{3*}\nu_3^{-1}, \id)$ is an isomorphism in the category 
$D(\cU_4) \rtimes MS(\cU_4),$ so after applying the functor $\tau_{\leq \bar n(3)}(\cdot, \cdot)$ we obtain an isomorphism
\begin{equation*}
	\nu_4:
	\tau_{\leq \bar n(3)}(Ri_{3*} \cP_3, \bW_{\bS}(Y_{n-3})) = \cP_4(\cL_4) \lra
	\tau_{\leq \bar n(3)}(Ri_{3*} \sc S\rest{\cU_3}, \bW_{\bS}(Y_{n-3})) \ \cong \sc S\rest{\cU_4},
\end{equation*}
where $\cL_4$ is $\bW_{\bS}(Y_{n-3})$ together with the injective sheaf map $\psi.$
Thus the statement holds over $\cU_4.$
One now continues to proceed as above to show that the statement holds over all of $\hat X.$
\end{proof}

\begin{definition}
We refer to the mezzoperversity $\cL$ in this theorem as {\bf the mezzoperversity of $\sc S$} and denote it  $\cL_{\bS}.$
\end{definition}

%%%%%%%%%%%%%%%%%%%%
\section{Global Duality}
%%%%%%%%%%%%%%%%%%%%

In this section we work with $\hat X$ a compact oriented topologically stratified pseudomanifold and with sheaves of $\bbR$-vector spaces.

In Proposition \ref{prop:FirstProp}($i$), we pointed out that if 
$\sc S \in D(\hat X)$ satisfies axioms [RP], and $\cD\sc S$ is its Verdier dual, then $(\cD\sc S) [-n]$ satisfies the axioms [RP] as well.
If $\cL$ is the mezzoperversity of $\sc S$ we refer to the mezzoperversity of $(\cD\sc S)[-n]$ as the {\bf dual mezzoperversity} of $\cL$ and denote it
$\sD \cL.$

Global duality for refined intersection homology groups is an immediate consequence upon taking hypercohomology.
Indeed, for any $\sc S \in D(\hat X),$ we have (see, e.g., \cite[\S 1.12]{GM2}, \cite[\S 4.4]{BanaglLong})
\begin{equation*}
	\bbH^j(\hat X; \sc S) 
	\cong 
	\Hom(\bbH^{n-j}(\hat X; \cD\sc S[-n]), \bbR)
\end{equation*}
and hence a non-degenerate bilinear pairing 
\begin{equation*}
	\bbH^j(\hat X; \sc S) \otimes 
	\bbH^{n-j}(\hat X; \cD\sc S[-n]) \lra \bbR.
\end{equation*}
In terms of mezzoperversities, this says that there is a natural non-degenerate pairing
\begin{equation*}
	\bbH^j(\hat X; \cP(\cL)) \otimes \bbH^{n-j}(\hat X; \cP(\sD \cL)) \lra \bbR.
\end{equation*}
%

%%%%%%%%%%%%%%%%%%%%
\section{Example: Self-dual sheaves} 
%%%%%%%%%%%%%%%%%%%%

In this section we work with oriented $n$-dimensional topologically stratified pseudomanifolds $\hat X$  and sheaves of $\bbR$-vector spaces.
As pointed out in \cite[\S1.9]{BanaglShort}, one could more generally work with vector spaces over any field of characteristic not equal to two.

\begin{theorem}
If $\sc S$ is a self-dual sheaf in $SD(\hat X)$ in the sense of \cite{BanaglShort} and $\sL$ is its Lagrangian structure, then $\sc S [-n]$ satisfies axioms [RP] and its mezzoperversity $\cL_{\bS}$ is naturally identified with $\sL.$
\end{theorem}

\begin{proof}
The category $SD(\hat X)$ is defined in \cite{BanaglShort} by four axioms ($SD1$)-($SD4$). 
Up to a degree shift by the dimension of $\hat X$, the first three axioms coincide with ($RP1$)-($RP3$) and by \cite[Lemma 2.1]{BanaglShort} elements of $SD(\hat X)$ also satisfy ($RP4$). Finally, both the Lagrangian structure and the mezzoperversity consist of the cohomology sheaves of $\sc S$ in degree $\bar n(k)$ over the strata of codimension $k$ for odd $k$ (together with the requisite monomorphisms). 
\end{proof}

Note that there are topological restrictions to finding objects in the category $SD(\hat X),$ for instance the signature of the links must vanish.
\begin{definition}
A {\em topological} stratified pseudomanifold $\hat X$ is an {\bf L-space} if $SD(\hat X) \neq \emptyset.$
\end{definition}

\begin{example}
The suspension $\hat X = \Sigma T^2$ of the $2$-torus $T^2$ is an L-space, as the signature
of $T^2$ vanishes.
The suspension $\hat X = \Sigma \mathbb{CP}^2$ of the complex projective space 
$\mathbb{CP}^2$ is not an L-space, as $\mathbb{CP}^2$ has signature $1$.
Note that both of these examples are indeed non-Witt spaces.
\end{example}

In the next section we show that the mezzoperversities in the analytic theory of \cite{ALMP13.1} coincide with the topological mezzoperversities described above.
In this setting of smoothly stratified spaces, a {\bf Cheeger space} is one admitting a self-dual analytic mezzoperversity.
Thus in the smooth setting, a self-dual mezzoperversity is equivalent to a Lagrangian structure in the sense of \cite{BanaglShort}, in particular we have
\begin{proposition}
A smoothly stratified pseudomanifold is a {\bf Cheeger space} iff it is an L-space.
\end{proposition}

Notice though that an L-space is a topologically stratified space, so the strata need not even be smoothable manifolds. Thus L-spaces constitute a larger class of spaces than Cheeger spaces.

%%%%%%%%%%%%%%%%%%%%
\section{Example: de Rham/Hodge cohomology of $\iie$ metrics} \label{sec:HodgeCoho}
%%%%%%%%%%%%%%%%%%%%

In this section we assume that all sheaves are real vector spaces and that the stratification is smooth.

A smooth stratification is the nomenclature introduced in \cite{ALMP11}. It coincides with the notion of 
{\em espace stratifi\'e} in \cite{BHS}, controlled pseudomanifold \cite{Pflaum}, and Thom-Mather stratified space \cite{Pollini}, see also \cite{Kloeckner}. Mather showed \cite[\S 8]{Mather} that any Whitney stratified subset of a smooth space, such as real and complex analytic varieties, admits Thom-Mather control data, hence a smooth stratification.
In \cite{ALMP11} it is shown that this class of spaces is equivalent to manifolds with corners endowed with {\em iterated fibration structures,} a class first defined by Richard Melrose.

\begin{definition}
A manifold with corners $\wt X$ has an {\bf iterated fibration structure} if every boundary hypersurface $H$ is the total space of a fibration 
\begin{equation*}
	\phi_H: H \lra Y_H
\end{equation*}
whose base $Y_H$ is itself a manifold with corners, and at each intersection $H_1 \cap H_2$ we have a commutative diagram
\begin{equation*}
	\xymatrix{ H_1 \cap H_2 \ar[rr]^{\phi_{H_1}} \ar[rd]_{\phi_{H_2} }
		& & \phi_{H_1}(H_1 \cap H_2) \ar[ld]^{\phi_{H_1H_2}} \\
		& Y_{H_2} & }
\end{equation*}
where all arrows are fibrations.
\end{definition}

Given a manifold with corners $\wt X$ with an iterated fibration structure, we obtain a (smoothly) stratified pseudomanifold by collapsing the fibrations in an appropriate order. Conversely, given a smoothly stratified pseudomanifold $\hat X,$ radial blow-ups of the strata in an appropriate order results in a manifold with corners with an iterated fibration structure, called the resolution of $\hat X.$ There is a natural map
\begin{equation*}
	\beta:\wt X \lra \hat X
\end{equation*}
that identifies the interior of $\wt X$ with the regular part of $\hat X,$ either of which we denote $X.$
For details, we refer the reader to \cite{ALMP11}.

There is a natural class of functions on $\wt X$ respecting the iterated fibration structure,
\begin{equation*}
	\CI_{\Phi}(\wt X) = \{ f \in \CI(\wt X): i_H^*f = \phi_H^*f_{Y_H} \},
\end{equation*}
where $i_H: H \hookrightarrow \wt X$ is the inclusion of the hyperplane $H$ and
$f_{Y_H} \in \CI (Y_H)$.
We view this as a natural replacement for the smooth functions on $\hat X,$ so we define
\begin{equation}\label{eq:SmoothFunHatX}
	\CI(\hat X) = \{ f \in \cC^0(\hat X) : \beta^*f\rest{X} \in \CI_{\Phi}(\wt X)\rest{X} \}.
\end{equation}

\begin{lemma} \label{lem:PartitionUnity} $ $
Given any open cover of a smoothly stratified space, $\hat X,$ there is a subordinate partition of unity in $\CI(\hat X).$
\end{lemma}

\begin{proof}
We follow the presentation of \cite[Chapter 2]{Lee:IntroSmooth} for smooth manifolds.
We start by defining the analogue of a `regular cover'.

Since $\hat X$ has finitely many strata, each of which is paracompact, $\hat X$ is itself paracompact.
(Let us assume, for simplicity of notation, that we have picked radial functions $r$ at each stratum so that $\{ r < 3 \}$ is a (Thom-Mather) tubular neighborhood.)
Let us briefly say that an open cover $\{ \cV_{\alpha} : \alpha \in \cA \}$ of $\hat X$ is {\em esteemed} if it is a countable, locally finite cover by precompact open sets such that: if $\cV_{\alpha} \subseteq \cU_2$ then there is a coordinate chart 
\begin{equation*}
	\phi_{\alpha}: \cV_{\alpha} \lra \bbR^n
\end{equation*}
such that $\phi_{\alpha}(\cV_{\alpha}) = \bbB_{3}(0),$ and we set $h(\alpha) = n;$
if $\cV_{\alpha}$ is not contained in $\cU_2$ then there is a singular stratum $Y_{h(\alpha)},$ a point $q_{\alpha} \in Y_{h(\alpha)},$ and a distinguished neighborhood chart
\begin{equation*}
	\phi_{\alpha}: \cV_{\alpha} \lra \bbR^{h(\alpha)} \times [0,3)_r \times Z_{q_{\alpha}}
\end{equation*}
such that $\phi_{\alpha}(\cV_{\alpha}) = \bbB_{3}(0) \times [0,3)_r \times Z_{q_{\alpha}};$
finally, we demand that the collection $\{ \wt \cV_{\alpha} : \alpha \in \cA \}$ given by
\begin{equation*}
	\wt \cV_{\alpha} =
	\begin{cases}
	 \phi_{\alpha}^{-1}(\bbB_1(0)) & \Mif \cU_{\alpha} \subseteq X \\
	 \phi_{\alpha}^{-1}(\bbB_{1}(0) \times [0,1)_r \times Z_{q_{\alpha}} ) & \text{ otherwise }
	\end{cases}
\end{equation*}
still covers $\hat X.$ Thus essentially a cover is esteemed if it is made up of distinguished open sets.

If $\hat X$ is smooth then an esteemed cover is just a regular cover, and every open cover has a regular refinement \cite[Proposition 2.24]{Lee:IntroSmooth}.
Using this fact at each stratum, it is easy to see that any open cover of $\hat X$ has an esteemed refinement.
It is also easy to show that an esteemed cover has a subordinate partition of unity in $\CI(\hat X)$:
For each $k<n,$ let $\rho_k: \bbR^k \times \bbR^+ \lra \bbR$ be a smooth non-negative function equal to one on $\bbB_1(0) \times [0,1)$ and supported in $\bbB_3(0) \times [0,3),$ also let $\rho_n: \bbR^n \lra \bbR$ be a smooth non-negative function equal to one on $\bbB_1(0)$ and supported in $\bbB_3(0).$
For each $\alpha$ let
\begin{equation*}
	\bar \psi_{\alpha}: \cV_{\alpha} \lra \bbR, \quad
	\bar \psi_{\alpha} = \rho_{h(\alpha)} \circ \phi_{\alpha}
\end{equation*}
where we let $\rho_{h(\alpha)}$ act on $\bbB_{3}(0) \times [0,3)_r \times Z_{q_{\alpha}}$ by disregarding the last coordinate.
Finally, the functions
\begin{equation*}
	\psi_{\alpha} = \frac{ \bar \psi_{\alpha} }{ \sum_{\alpha \in \cA} \bar \psi_{\alpha} }
\end{equation*}
are a partition of unity in $\CI(\hat X)$ subordinate to $\{ \cV_{\alpha} : \alpha \in \cA \}.$

\end{proof}

There is a class of metrics that reflects the conic structure of a stratified space near a singular stratum, the {\bf iterated incomplete edge, or $\iie,$ metrics} (also known as iterated wedge metrics). First consider $\hat X$ with a total of two strata, say with filtration $Y \subseteq \hat X,$ and let $x$ be a {\bf boundary defining function} for $Y$ meaning a non-negative function $x \in \CI(\hat X)$ such that $Y= x^{-1}(0)$ and $|dx|$ is bounded away from zero as $x\to 0.$
An $\iie$ metric on $\hat X$ is a Riemannian metric on $\hat X^{\reg} = \cU_2$ that in a tubular neighborhood of $Y$ has the form
\begin{equation}\label{eq:LocalMet}
	dx^2 + x^2 g_Z + \phi^*g_Y
\end{equation}
where $g_Y$ is a Riemannian metric on $Y$ and $g_Z$ restricts to the fibers over $Y$ to be a Riemannian metric on the link.
An $\iie$ metric on a general stratified space $\hat X$ is defined, inductively over the depth of $\hat X,$ as a Riemannian metric on $\hat X^{\reg}$ that near each singular stratum has the form \eqref{eq:LocalMet} where now $g_Z$ is itself an $\iie$ metric on the stratified space $Z.$
As an $\iie$ metric is defined on the regular part of $\hat X$ we can equally well consider it as a metric on the interior of $\wt X.$
Any smoothly stratified space can be endowed with an $\iie$ metric. (In our constructions below we will assume that we are working with a {\em rigid, suitably scaled} $\iie$ metric, these metrics also exist on any stratified space as shown in \cite{ALMP11} to which we refer for the definitions and details.)

On a singular manifold it is often useful to replace the (co)tangent bundle with a bundle adapted to the geometry. 
For a stratified space with an $\iie$ metric, $g,$ a natural choice is the $\iie$-cotangent bundle, $\Iie T^*X,$ defined over $\wt X$ as those one-forms whose restriction to each boundary hypersurface vanishes on vertical vector fields.
Thus sections of $\Iie T^*X$ are locally spanned by
\begin{equation*}
	dx, \quad dy, \quad x\; dz.
\end{equation*}
We can also identify $\Iie T^*X$ as the bundle whose sections are one-forms on $\wt X$ with bounded pointwise length with respect to $g.$
This makes it easy to see that the metric $g$ induces a bundle metric on $\Iie T^*X.$
An advantage of $\Iie T^*X$ over $T^*\wt X$ is that forms like $x \; dz$ that vanish at $x=0$ as sections of $T^*\wt X$ do not vanish at $x=0$ as sections of $\Iie T^*X,$ which better reflects the structure of the space.
The metric on $\Iie T^*X$ allows us to define the space of $L^2$ $\iie$-differential forms,
\begin{equation*}
	L^2(\wt X; \Lambda^*(\Iie T^*X)) = L^2(\hat X^{\reg}; \Lambda^* T^*\hat X^{\reg} ),
\end{equation*}
where the equality comes from identifying an $L^2$-form on $\wt X$ with its restriction to 
the interior of $\wt X.$

The exterior derivative on $\hat X^{\reg}$ extends from $\hat X^{\reg}$ to a differential operator on $\wt X$ acting on $\iie$-differential forms
\begin{equation}\label{eq:Initiald}
	d: \CIc(\wt X; \Lambda^j (\Iie T^*X) )
	\lra \CIc(\wt X; \Lambda^{j+1} (\Iie T^*X) )
\end{equation}
where we will use $\CIc$ to denote forms that are compactly supported in the interior of $\wt X$ or equivalently the regular part of $\hat X.$
The space $\CIc(\wt X; \Lambda^j (\Iie T^*X) )$ is dense in $L^2(\wt X; \Lambda^*(\Iie T^*X))$ so the operator \eqref{eq:Initiald} has two canonical extensions to a closed operator on the latter; the minimal extension of $d$ is the operator $d$ with domain
\begin{equation*}
	\cD_{\min}(d) = \{ \omega \in L^2(\wt X; \Lambda^*(\Iie T^*X)) : 
	\exists (\omega_n) \subseteq \CIc(\wt X; \Lambda^j (\Iie T^*X) ) \Mst \omega_n \to \omega \Mand (d\omega_n) \text{ Cauchy} \}
\end{equation*}
where we set $d\omega = \lim d\omega_n,$
and the maximal extension of $d$ is the operator with domain
\begin{equation*}
	\cD_{\max}(d) = \{ \omega \in L^2(\wt X;\Lambda^*(\Iie T^*X)) : \text{ $d\omega,$ computed distributionally, is in } L^2(\wt X;\Lambda^*(\Iie T^*X)) \}.
\end{equation*}
All closed extensions $(d, \cD(d))$ of \eqref{eq:Initiald} satisfy
\begin{equation*}
	\cD_{\min}(d) \subseteq \cD(d) \subseteq \cD_{\max}(d).
\end{equation*}

The formal adjoint of $d,$ 
\begin{equation}\label{eq:Initialdelta}
	\delta: \CIc(\wt X; \Lambda^j (\Iie T^*X) )
	\lra \CIc(\wt X; \Lambda^{j-1} (\Iie T^*X) ),
\end{equation}
has analogously defined domains $\cD_{\min}(\delta)$ and $\cD_{\max}(\delta)$ and these satisfy
\begin{equation*}
	( d, \cD_{\min}(d) )^* = (\delta, \cD_{\max}(\delta)), \quad
	( d, \cD_{\max}(d) )^* = (\delta, \cD_{\min}(\delta)).
\end{equation*}
$ $

We now recall the natural class of domains for $d$ studied in  \cite{ALMP13.1}.
Consider first a stratified space $\hat X$ with a total of two strata, with filtration $Y \subseteq \hat X,$ as above.
Let $H$ be the boundary hypersurface lying above $Y$ in the resolution of $\hat X,$ $\wt X.$
If $\hat X$ is a Witt space, i.e., if $Y$ has even codimension or the vertical (Hodge) cohomology groups\footnote{$\cH^{\mid}(Z)$ refers to space of harmonic forms on $Z$ with respect to the induced metric, of differential form degree equal to half of the dimension of $Z.$} $\cH^{\mid}(Z)$ vanish, then 
\begin{equation*}
	\cD_{\max}(d) = \cD_{\min}(d).
\end{equation*}
Otherwise, given $\omega \in \cD_{\max}(d),$ let $\omega_{\delta}$ be its orthogonal projection off of $\ker (\delta, \cD_{\min}(\delta)).$
It is shown in {\em loc.cit.} that $\omega_{\delta}$ has a (distributional) asymptotic expansion at the boundary with leading term 
$\alpha(\omega_{\delta}) + dx \wedge \beta(\omega_{\delta})$ and that $\omega \in \cD_{\min}(d)$ if and only if 
$\alpha(\omega_{\delta}) = \beta(\omega_{\delta}) = 0.$
Thus we can define domains for $d$ by imposing conditions on this leading term.
The vertical cohomology bundle $\cH^{\mid}(H/Y) \lra Y$ has a natural flat structure, 
which can equivalently be thought of as coming from the exterior derivative on $H$ or from the homotopy invariance of the vertical de Rham cohomology.
If $W \subseteq \cH^{\mid}(H/Y)$ is a flat subbundle, then we define a domain corresponding to {\bf Cheeger ideal boundary conditions} from $W$ by
\begin{equation*}
	\cD_{W}(d) = \{ \omega \in \cD_{\max}(d) : \alpha(\omega_{\delta}) \text{ is a distributional section of } W \}.
\end{equation*}
In \cite{ALMP13.1} it is shown that $(d, \cD_W(d))$ is a closed operator and that filtering by differential form degree yields a Fredholm de Rham complex, for any such $W.$

For a general stratified space, we define domains for $d$ inductively.
If $\hat X$ is Witt, then $\cD_{\min}(d) = \cD_{\max}(d).$ 
Otherwise given a form $\omega \in \cD_{\max}(d)$ and $\omega_{\delta}$ its projection off of $\ker(\delta, \cD_{\min}(\delta)),$
$\omega$ has a distributional expansion at\footnote{If $Y_{n-3} = \emptyset$ then there is no need to impose a boundary condition here, for simplicity of exposition we  assume for the moment that all singular strata of  odd codimension are non-empty and non-Witt.} $Y_{n-3}$ whose leading term we denote
\begin{equation*}
	\alpha_{n-3}(\omega_{\delta}) + dx\wedge\beta_{n-3}(\omega_{\delta}).
\end{equation*}
We can impose Cheeger ideal boundary conditions at $Y_{n-3}$ by choosing a flat subbundle $W(Y_{n-3})$ of the middle-degree vertical Hodge bundle over $Y_{n-3},$ 
\begin{equation*}
	\cH^{\mid}(Z_{n-3}) - \cH^{\mid}(H_{n-3}/Y_{n-3}) \lra Y_{n-3},
\end{equation*}
where $H_{n-3}$ is the boundary hypersurface of $\wt X$ lying above $Y_{n-3},$
and defining
\begin{equation*}
\begin{gathered}
	\wt\cL_3 = \{ W(Y_{n-3}) \}\\
	\cD_{\max, \wt\cL_3}(d) 
	= \{ \omega \in \cD_{\max}(d) : \alpha(\omega_{\delta}) \text{ is a distributional section of } W(Y_{n-3}) \}.
\end{gathered}
\end{equation*}
This domain defines a closed extension of $d$ with adjoint $(\delta, \cD_{\min, \wt\cL_3}(\delta)).$

The next odd codimensional stratum is  $Y_{n-5}$ and we denote the link of $\hat X$ at this stratum by $Z_{n-5}.$ This link is itself 
a stratified space; it has a  single singular stratum the link of which is $Z_{n-3},$ so the boundary conditions $\wt\cL_3$ imposed on the exterior derivative of $\hat X$ at $Y_{n-3}$ induce boundary conditions $\wt\cL_3(Z_{n-5})$ for the exterior derivative of $Z_{n-5}.$ 
The Hodge bundle over $Y_{n-5}$ with these boundary conditions, 
\begin{equation*}
	\cH^{\mid}_{\wt\cL_3(Z_{n-5})}(Z_{n-5}) - \cH^{\mid}_{\wt\cL_3(Z_{n-5})}(H_{n-5}/Y_{n-5}) \lra Y_{n-5},
\end{equation*}
inherits a flat structure and to impose boundary conditions at $Y_{n-5}$ we choose a flat subbundle $W_{n-5}.$
Unfortunately, the form $\omega_{\delta}$ above need not have an asymptotic expansion at $Y_{n-5}.$
However, if $\omega \in \cD_{\max, \wt\cL_3}(d)$ and $\omega_{\delta, \wt\cL_3}$ is its orthogonal projection off of $\ker (\delta, \cD_{\min, \wt\cL_3}(\delta))$
then $\omega_{\delta, \wt\cL_3}$ has a distributional asymptotic expansion at $Y_{n-5},$ with leading term denoted
\begin{equation*}
	\alpha(\omega_{\delta, \wt\cL_3}) + dx\wedge \beta(\omega_{\delta, \wt\cL_3}).
\end{equation*}
We define
\begin{equation*}
\begin{gathered}
	\wt\cL_5 = \{ W(Y_{n-3}), W(Y_{n-5}) \} \\
	\cD_{\max, \wt\cL_5}(d)
	= \{ \omega \in \cD_{\max}(d) : \alpha(\omega_{\delta}) \text{ is a distributional section of $W(Y_{n-3})$ over $Y_{n-3}$} \\
	\Mand \alpha(\omega_{\delta, \wt\cL_3})  \text{ is a distributional section of $W(Y_{n-5})$ over $Y_{n-5}$} \}.
\end{gathered}
\end{equation*}

We continue in this way, iteratively choosing and imposing boundary conditions at each non-Witt stratum of even codimension, until eventually we can define, with $\ell=0$ if $n$ is odd, and else $\ell=1,$
\begin{equation*}
\begin{gathered}
	\wt\cL = \{ W(Y_{n-3}), W(Y_{n-5}), \ldots, W(Y_{\ell}) \}\\
	\cD_{\wt\cL}(d)
	= \{ \omega \in \cD_{\max}(d) : 
	\alpha(\omega_{\delta}) \text{ is a distributional section of $W(Y_{n-3})$ over $Y_{n-3},$} \\
	\alpha(\omega_{\delta, \wt\cL_3})  \text{ is a distributional section of $W(Y_{n-5})$ over $Y_{n-5},$ } \ldots,\\
	\alpha(\omega_{\delta, \wt\cL_{n-\ell-2}})  \text{ is a distributional section of $W(Y_{\ell})$ over $Y_{\ell}$} \}.
\end{gathered}
\end{equation*}
We refer to these boundary conditions as {\bf Cheeger ideal boundary conditions} corresponding to the analytic mezzoperversity $\wt\cL.$
It is shown in \cite{ALMP13.1} that $(d, \cD_{\wt\cL}(d))$ is a closed operator and filtering by differential form degree yields a Fredholm de Rham complex.\\

A useful fact is that this domain is closed under multiplication by functions in the space $\CI(\hat X)$ from \eqref{eq:SmoothFunHatX}.
Indeed, if $\omega \in \cD_{\wt\cL}(d)$ and $f \in \CI(\hat X)$ then $f\omega \in L^2(\wt X;\Lambda^*(\Iie T^*X))$
and 
\begin{equation*}
	d(f\omega) = df \wedge \omega + f d\omega
\end{equation*}
is also in $L^2(\wt X;\Lambda^*(\Iie T^*X))$ since $df$ is an $\iie$ one-form;
moreover it is easy to see that the leading term in the expansion of $(f\omega)_{\delta, \wt\cL_k}$ at a stratum $Y_{n-k-2}$ is the leading term of $\omega_{\delta, \wt\cL_k}$ multiplied by $f\rest{Y_{n-k-2}} \in \CI(Y_{n-k-2}),$ so $\omega \in \cD_{\wt\cL}(d)$ implies $f\omega \in \cD_{\wt\cL}(d).$

\begin{definition}
Let us define a pre-sheaf on $\hat X$ by assigning to each open set $\cU \subseteq \hat X$ the vector space
\begin{equation*}
	\cD_{\wt\cL}(\cU) := \{ \omega \in \cD_{\wt\cL}(d): \supp \omega \subseteq ( \cU \cap \hat X^{\reg}) \}
\end{equation*}
and assigning to each inclusion $j:\cV \hookrightarrow \cU$ of open sets the restriction map
\begin{equation*}
	j^*: \cD_{\wt\cL}(\cU) \lra \cD_{\wt\cL}(\cV).
\end{equation*}
This presheaf is filtered by differential form degree and the exterior derivative makes it into a complex of presheaves.
We sheafify and denote the corresponding sheaf complex by $\bL^2_{\wt\cL}\sc\bOm \in \Sh^{\bullet}(\hat X).$
\end{definition}

It is worth pointing out that sheafification is necessary here. The presheaf $\cU \mapsto \cD_{\wt\cL}(\cU)$ is not a sheaf as it does not satisfy the unique gluing condition.
This is true even at the level of zero-forms. Indeed, if $x$ is a function that vanishes linearly on a singular stratum and $\{ \cV_{\alpha} : \alpha \in \sA \}$ is an open cover of the regular part of $\hat X$ such that each $\bar {\cV}_{\alpha} \subseteq \hat X^{\reg},$ then the sections
\begin{equation*}
	\tfrac1{x^2} \in \cD_{\wt\cL}(\cV_{\alpha} )
\end{equation*}
coincide on overlaps. However as $\tfrac1{x^2}$ is not an $L^2$-function on $\hat X^{\reg},$ these local sections do not glue together to an element of $\cD_{\wt\cL}(\hat X).$ Thus in general, for $\cU \subseteq \hat X,$
\begin{equation*}
	\Gamma( \cU, \bL^2_{\wt\cL}\sc\bOm ) \neq \cD_{\wt\cL}(\cU).
\end{equation*}
This makes the following result remarkable.

\begin{lemma} \label{lem:SheafDescription}$ $
\begin{itemize}
\item [1)]
The  space of global sections of the sheaf $\bL^2_{\wt\cL}\sc\bOm$ coincides with the  space of global sections of the presheaf $\cU \mapsto \cD_{\wt\cL}(\cU),$
\begin{equation*}
	\Gamma( \hat X, \bL^2_{\wt\cL}\sc\bOm ) = \cD_{\wt\cL}(\hat X).
\end{equation*}
\item [2)]
The sheaf $\bL^2_{\wt\cL}\sc\bOm$ satisfies
\begin{equation*}
	\Gamma(\cU, \bL^2_{\wt\cL}\sc\bOm) 
	= \{ \omega \in L^2_{\loc}(\cU): 
	\text{every } p \in \cU \text{ has a neighborhood $\cV \subseteq \cU$ such that } \omega\rest{\cV} \in \cD_{\wt\cL}(\cV) \}. 
\end{equation*}

\end{itemize}
\end{lemma}

\begin{proof}
We follow \cite[\S 4]{Bei}.\\

We prove ($2$) first. Note that the presheaf
\begin{equation*}
	\cU\mapsto \cE_{\wt\cL}(\cU)
	= \{ \omega \in L^2_{\loc}(\cU): 
	\text{every } p \in \cU \text{ has a neighborhood $\cV \subseteq \cU$ such that } \omega\rest{\cV} \in \cD_{\wt\cL}(\cV) \}
\end{equation*}
is in fact a sheaf, and the natural inclusions
\begin{equation*}
	\cD_{\wt\cL}(\cU) \lra \cE_{\wt\cL}(\cU)
\end{equation*}
form a map of presheaves and so induce a map of sheaves, $F: \bL^2_{\wt\cL}\sc\bOm \lra \cE_{\wt\cL}.$
It suffices to see that, for every $x \in \hat X,$ the induced map on stalks $F_x: \bL^2_{\wt\cL}\sc\bOm_x \lra (\cE_{\wt\cL})_x$ is an isomorphism which we show as follows:
Assume $F_x(s)$ is zero. Then it has a representative $\omega \in \cE_{\wt\cL}(\cU)$ with $\omega=0$ in a neighborhood $\cV$ of $x.$
Thus $0 \in \cD_{\wt\cL}(\cV)$ represents $s,$ i.e.,  $s=0,$ so $F_x$ is injective.
Next, if $t \in (\cE_{\wt\cL})_x$ is represented by  $\eta \in \cE_{\wt\cL}(\cU)$ then there is a neighborhood $\cV$  of $x$ in $\cU$ such that $\eta\rest{\cV} \in \cD_{\wt\cL}(\cV).$ So the germ of $\eta\rest{\cV}$ at $x,$ $s \in \bL^2_{\wt\cL}\sc\bOm_x,$ satisfies $F_x(s) =t.$\\

Now to prove ($1$), we start with the canonical map
\begin{equation*}
	\phi: \cD_{\wt\cL}(\hat X) \lra 
	\Gamma( \hat X, \bL^2_{\wt\cL}\sc\bOm(\hat X) ) 
\end{equation*}
sending an $L^2$-differential form $\omega$ to the section 
\begin{equation*}
	\hat X \ni x \mapsto \omega_x \in \bL^2_{\wt\cL}\sc\bOm_x,
\end{equation*}
where $\omega_x$ is the germ of $\omega$ at $x.$
In particular, $\phi(\omega)(x) = \phi(\eta)(x)$ means that there is an open set in $\hat X$ containing $x$ on which $\omega = \eta.$
Thus if $\phi(\omega) = 0$ then there are open sets covering $\hat X$ on which $\omega$ vanishes and hence $\omega =0,$ so $\phi$ is injective.

Having established ($2$), we know that
\begin{equation*}
	\Gamma( \hat X, \bL^2_{\wt\cL}\sc\bOm(\hat X) )
	= \{ \omega \in L^2_{\loc}(\hat X): 
	\text{every } p \in \hat X \text{ has a neighborhood $\cV \subseteq \hat X$ such that } \omega\rest{\cV} \in \cD_{\wt\cL}(\cV) \}.
\end{equation*}
Let $\omega \in \Gamma( \hat X, \bL^2_{\wt\cL}\sc\bOm(\hat X) )$ and for each $x \in \hat X$ let $\cV_x$ be an open neighborhood such that $\omega \in \cD_{\wt\cL}(\cV_x).$ These open sets cover $\hat X,$ so by compactness there is a finite sub cover
\begin{equation*}
	\hat X = \cV_{x_1} \cup \ldots \cup \cV_{x_N}.
\end{equation*}
By Lemma \ref{lem:PartitionUnity} there is a partition of unity $\{ \psi_1, \ldots, \psi_N \} \subseteq \CI_{\Phi}(\wt X)$ subordinate to  the cover
$\{  \beta^{-1}\cV_{x_1}, \ldots, \beta^{-1}\cV_{x_N} \}$ of $\wt X.$ We have
\begin{equation*}
	\omega = \sum \psi_j \omega\rest{\cV_{x_j}} \text{ over } \hat X^{\reg}
\end{equation*}
and since each summand is in $\cD_{\wt\cL}(\hat X),$ so is $\omega.$ This shows that $\phi$ is surjective.
\end{proof}

From Lemma \ref{lem:SheafDescription}(2), the fact that $\cD_{\wt\cL}$ is closed under multiplication by elements of $\CI(\hat X)$ implies that sections of $\bL^2_{\wt\cL}\sc \Omega$ are closed under multiplication by elements of $\CI(\hat X).$
A useful consequence is that the sheaf $\bL^2_{\wt\cL}\sc \Omega$ is 
soft -- that is, every section of $\bL^2_{\wt\cL}\sc \Omega$ on a closed set extends to a section of $\bL^2_{\wt\cL}\sc \Omega$ over all of $\hat X.$
Indeed, if $K \subseteq \hat X$ is a closed set and $s \in \bL^2_{\wt\cL}\sc \Omega(K)$ then $s$ can be represented by a section $\wt s \in \bL^2_{\wt\cL}\sc \Omega(\cU)$ for some neighborhood $\cU$ of $K$ and by Lemma \ref{lem:PartitionUnity} we can find $f \in \CI(\hat X)$ equal to one on $K$ and vanishing outside of $\cU.$
Thus $fs$ is an extension of $s$ from $K$ to a section of $\bL^2_{\wt\cL}\sc \Omega(\hat X).$

\begin{lemma} \label{lem.constr}
The complex of sheaves $\bL^2_{\wt\cL}\sc\Omega$ is constructible with respect to the stratification of $\hat X.$
\end{lemma}

\begin{proof}
Recall \cite[\S 1.4]{GM2} that a complex of sheaves is constructible with respect to a given stratification if its restriction to each stratum has locally constant cohomology sheaves, with finitely generated stalks. 
Let $\cU$ be a distinguished neighborhood of a point $p \in \hat X.$
If $p \in \hat X^{\reg},$ then $\cU$ is a ball and by the Poincar\'e Lemma its reduced de Rham cohomology vanishes; otherwise if $p$ is on a singular stratum $Y,$ $\cU \cong \bbB^h \times C(Z),$ and the cohomology was computed in \cite[\S 7]{ALMP13.1}: for a fixed flat trivialization $W(Y)\rest{\cU} \cong \cU \times W(Y)_p,$ we have the generalized Poincar\'e Lemma
\begin{equation}\label{eq:PoincareLemma}
	\Ht^k(\Gamma(\cU, \bL^2_{\wt\cL}\sc\Omega))
	=\begin{cases}
	\cH^k_{\wt\cL(Z)}(Z) & \Mif k < \tfrac12\dim Z \\
	W(Y)_p & \Mif k = \tfrac12\dim Z\\
	0 & \Mif k >\tfrac12 \dim Z
	\end{cases}
\end{equation}
In either case we note that all distinguished neighborhoods of a given point have the same cohomology which hence coincides with the stalk cohomology.
Since $\cU$ is a distinguished neighborhood of all points in $\cU \cap Y,$ the local cohomology of $\bL^2_{\wt\cL}\sc \Omega$ is locally constant.
Since the groups 
$\cH^k_{\wt\cL(Z)}(Z),$ $W(Y)_p$ are    
finitely generated, the sheaf complex $\bL^2_{\wt\cL}\sc \Omega$ is constructible.
\end{proof}

Since soft resolutions can be used to compute $Ri_{k*},$ we have $Ri_{k*}\bL^2_{\wt\cL}\sc \Omega = i_{k*} \bL^2_{\wt\cL}\sc \Omega$ for all $k.$
To check the stalk vanishing condition of perversity $\bar p$ for $\sc A,$ the sheafification of a presheaf 
$A^\bullet$  at the stratum $Y_{n-k},$ it is enough to show that
\begin{equation*}
	\Ht^j(  A^{\bullet} (\cU) ) = 0  \Mif j>\bar p(k)
\end{equation*}
for all distinguished neighborhoods $\cU$ of points in $Y_{n-k}.$
Indeed, we have
\begin{equation}\label{eq:StalkCohoComp}
	\bbH^j(\sc A)_x = H^j( \sc A_x) = H^j( \lim_{\cU\ni x} A^{\bullet}(\cU) )
	=  \lim_{\cU \ni x} H^j( A^{\bullet}(\cU) ).
\end{equation}

\begin{theorem}\label{thm:MainThm}
Let $(\hat X,g)$ be a smoothly stratified pseudomanifold with a suitably scaled $\iie$-metric and let $\sc{\wt\cL}$ be an analytic mezzoperversity. 
The sheaf $\bL^2_{\wt\cL}\sc\bOm$ is in $RP(\hat X).$
\end{theorem}

\begin{proof}
The constructibility of the sheaf complex $\bL^2_{\wt\cL}\sc\bOm$ is established
in Lemma \ref{lem.constr}.
We check that $\bL^2_{\wt\cL}\sc\bOm$ satisfies the axioms ($RP1$)-($RP4$). 
The normalization and lower bound axioms are automatic and, by the comments above, ($RP3$) follows from  the local Poincar\'e Lemma \eqref{eq:PoincareLemma},
so we only need to check ($RP4$).

To check ($RP4$) first let us note that each $\cU_k$ is locally compact and countable at infinity.
(The latter means that it can be written as a countable union of compact sets, $(K_n),$ with $K_\ell \subseteq K_{\ell+1}^{\circ}.$ One can take the sub-level sets $K_{\ell} = \{ x \geq \tfrac1\ell \}$ of a boundary defining function $x$ for $Y_{n-k}$ as the family of compact sets.)
We know that $\sc S = \bL^2_{\wt\cL}\sc\Omega$ is soft and hence $c$-soft.
So $i_k^*\sc S$ is a $c$-soft sheaf on $\cU_k,$ and by \cite[\S IV.2, Corollary 2.3]{Iversen}, a $c$-soft sheaf on a locally compact set countable at infinity is soft.
Now since soft resolutions can be used to compute $Ri_{k*},$ we have $Ri_{k*}i_k^*\sc S = i_{k*}i_k^*\sc S.$
Then using \eqref{eq:StalkCohoComp} we have at a point $x\in Y_{n-k}$:
\begin{equation*}
	H^j(i_{k*}i_k^*\sc S)_x 
	= \lim_{\cV \ni x} H^j\Gamma(i_k^{-1}(\cV), i_k^*\sc S)
	= \lim_{\cV \ni x} H^j\Gamma(\cV \cap \cU_k, i_k^*\sc S).
\end{equation*}
Next notice that, for any distinguished neighborhood $\cV$ of $x$ we have
\begin{multline*}
	\Gamma(\cV \cap \cU_k, i_k^*\sc S) 
	= \{ \omega \in L^2_{\loc}(\cV \cap \cU_k) : \\
	\text{every } p \in \cV \cap \cU_k \text{ has a neighborhood $\cW \subseteq \cV\cap \cU_k$ such that } \omega\rest{\cW} \in \cD_{\wt\cL}(\cW) \}. 
\end{multline*}
Here the domain $\cV \cap \cU_k$ can be identified with $\bbB^h \times ( C(Z_{n-k}) \setminus x ) \cong \bbB^h \times (0,1) \times Z_{n-k}$ and hence this is the $L^2$-cohomology of $\bbB^h \times (0,1) \times Z_{n-k}$ with boundary conditions $\wt \cL_{k-1}(Z_{n-k}).$
Thus by the local computation above we have
\begin{equation*}
	H^j\Gamma(\cV \cap \cU_k, i_k^*\sc S)
	= \cH^j_{\wt \cL_{k-1}(Z_{n-k})}(Z_{n-k})
\end{equation*}
The adjunction map from $\sc S_x$ to $Ri_{k*}i_k^*\sc S_x$ corresponds to the restriction map $r$ from $\cV$ to $\cV \cap \cU_k,$
and the computations above show that 
\begin{equation*}
	r: 
	H^j\Gamma(\cV, \sc S)
	\lra
	H^j\Gamma(\cV \cap \cU_k, i_k^*\sc S)
\end{equation*}
is an isomorphism for $j \leq \bar m(k),$ so we have verified axiom $(RP4).$
\end{proof}

Thus for any analytic mezzoperversity $\wt\cL$ we now know from Theorem \ref{thm:DeligneSheaf} that there exists a topological mezzoperversity $\cL$ such that
\begin{equation}\label{eq:FirstEquality}
	\bL^2_{\wt\cL}\sc\Omega \text{ is quasi-isomorphic to } \cP_{\cL}. 
\end{equation}
The topological mezzoperversity $\cL$ consists of the degree $\bar n(k)$-cohomology sheaves of $\bL^2_{\wt\cL}\sc\Omega$ over the non-Witt singular strata, and by the local computation \eqref{eq:PoincareLemma}, this is the sheaf of flat sections of the flat bundle $W(Y_{n-k}).$
Thus we can identify an analytic mezzoperversity with a topological one by replacing each flat bundle with its sheaf of flat sections.

Conversely, given a topological mezzoperversity we show that it can be realized as an analytic mezzoperversity.
It suffices to work with the Deligne sheaf of a topological mezzoperversity $\cL$ and show that we can identify it with $\bL^2_{\wt\cL}\sc\Omega$ for some analytic mezzoperversity $\wt\cL.$
We construct $\bL^2_{\wt\cL}\sc \Omega_k \in D(\cU_k)$ inductively over $k.$
Over $\cU_3$  we have an isomorphism 
\begin{equation*}
	\cP(\cL)\rest{\cU_3} = \cP_3(\cL) \cong \bL^2\sc\Omega_3
\end{equation*}
where $\bL^2\sc\Omega_3$ is the $L^2$ sheaf without any boundary conditions.
Recall that $\cP_4(\cL)$ is constructed over $\cU_4$ by
\begin{equation*}
	\cP_4(\cL_4) 
	= \tau_{\leq \bar n(3)} (Ri_{3*}\cP_3, \bW(Y_{n-3})),
\end{equation*}
where $\bW(Y_{n-3})$ is a subsheaf of $\bH^{\bar n(3)}(Ri_{3*}\cP_3)$ over $Y_{n-3}.$
But $\cP_3$ is quasi-isomorphic to $\bL^2\sc\Omega_3,$
and we can identify the sheaf
\begin{equation*}
	\bH^{\bar n(3)}(Ri_{3*} \bL^2\sc\Omega_3)\rest{Y_{n-3}}
\end{equation*}
with the sheaf of flat sections of the flat bundle $\cH^{\mid}(H_{n-3}/Y_{n-3}).$ Indeed, the sections of the former sheaf over any distinguished neighborhood are the flat sections of $\cH^{\mid}(H_{n-3}/Y_{n-3})$ (by the local Poincar\'e Lemma).
Thus $\bW(Y_{n-3})$ is a subsheaf of 
the locally constant sheaf of flat sections of $\cH^{\mid}(H_{n-3}/Y_{n-3}) \lra Y_{n-3}.$ This allows us to identify, for some flat sub-bundle 
$W(Y_{n-3}) \subseteq \cH^{\mid}(H_{n-3}/Y_{n-3}),$
\begin{equation*}
	\bW(Y_{n-3}) = \text{ flat sections of } W(Y_{n-3}).
\end{equation*}
Now over $\cU_4$ we have both $\cP_4(\cL_4)$ and $\bL^2_{\wt\cL_4}\sc \Omega,$ where $\wt\cL_4 = \{ W(Y_{n-3}) \}.$
As mentioned in \eqref{eq:FirstEquality}, since the topological perversity $\cL_4$ corresponds to the analytic perversity $\wt\cL_4,$ we have
\begin{equation*}
	\cP_4(\cL_4) \cong \bL^2_{\wt\cL_4}\sc \Omega \text{  over }\cU_4.
\end{equation*}
It is now clear that one can proceed inductively and show that each $\bW(Y_{n-k}) \in \cL$ can be identified with the sheaf of flat sections of a flat sub-bundle $W(Y_{n-k})$ of
\begin{equation*}
	\cH^{\mid}_{\wt\cL_k(Z_{n-k})}(H_{n-k}/Y_{n-k}) \lra Y_{n-k}.
\end{equation*}
We have thus shown:

\begin{theorem}
On a compact smoothly stratified pseudomanifold $\hat X$, every topological perversity $\cL$ corresponds to an analytic perversity $\wt\cL.$
The hypercohomology of the Deligne sheaf of $\cL$ can be computed as the $L^2$-de Rham cohomology with Cheeger ideal boundary conditions induced by $\wt\cL,$
\begin{equation*}
	\bbH^j(\hat X; \cP_{\cL}) = \bbH^j(\hat X; \bL^2_{\wt\cL}\sc\Omega)
	= H^j( \cD_{\wt\cL}(d))
	= \Ht^j_{\wt\cL}(\hat X).
\end{equation*}
\end{theorem}
Here the second equality comes from Lemma \ref{lem:SheafDescription} and the final one is a definition.

\end{document}